\documentclass[12pt]{article}

\usepackage{amsthm}

\usepackage[left=2cm,right=2cm,top=2cm,bottom=2cm,bindingoffset=0cm]{geometry}

\usepackage[T2A]{fontenc}
\usepackage[utf8]{inputenc}
\usepackage[english]{babel}
\usepackage{amsmath}
\usepackage{amsfonts}
\usepackage{mathtools}
\usepackage{graphicx}
\usepackage{systeme}
\usepackage{tikz}
\usepackage{blindtext}
\usepackage{amssymb}
\usepackage{hyperref}
\usepackage{bbm}
\usepackage{cite}
\usepackage[normalem]{ulem}
\graphicspath{ {./} }

\newtheorem{theorem}{Theorem}[section]
\newtheorem{corollary}{Corollary}[section]
\newtheorem{lemma}{Lemma}[section]
\newtheorem*{lemma*}{Lemma}
\newtheorem*{den*}{Denotion}

\newtheorem*{remark*}{Remark}
\newtheorem{construction}{Construction}

\theoremstyle{remark}
\newtheorem*{remark}{Remark}

\theoremstyle{definition}
\newtheorem{definition}{Definition}[section]

\newcommand{\mP}{\mathbf{P}}

\newcommand{\toas}{\xrightarrow{\text{a.s.}}}

\title{Modularity in planted partition model\footnote{This work was supported by a grant for research centers in the field of artificial intelligence, provided by the Analytical Center for the Government of the Russian Federation in accordance with the subsidy agreement (agreement identifier 000000D730321P5Q0002) and the agreement with the Moscow Institute of Physics and Technology dated November 1, 2021 No. 70-2021-00138.}}
\author{Mikhail Koshelev\footnote{Moscow Institute of Physics and Technology, Moscow State University, mkoshelev99@gmail.com}}
\date{}

\begin{document}

\maketitle

\begin{abstract}

We obtain tight bounds on the modularity of PPM graphs in the case of equally sized parts. Moreover, we provide a general method that can help in obtaining bounds for various other models.

\end{abstract}

\section{Introduction}

In this paper we investigate the modularity of the planted partition model (PPM, sometimes also referred to as the stochastic block model). Let us give a formal definition.

\begin{definition}
    For integers $k, n_1, \dotsc, n_k$ and $p, q \in [0, 1]$ consider a graph on vertices $$\{v_{1, 1}, \dotsc, v_{1, n_1}, v_{2, 1}, \dotsc, v_{k, n_k}\},$$ where we draw each edge $(v_{i_1, j_1}, v_{i_2, j_2})$ independently from each other with probability $p$ if $i_1 = i_2$ and probability $q$ in the other case. The sets of vertices $\{v_{i, 1}, \dotsc, v_{i, n_i}\}$ are often referred to as clusters.
\end{definition}

It is worth noting that there is a similar model, which is also sometimes referred to as the block model (or stochastic block model). Instead of fixing the sizes of the clusters, $n$ vertices are assigned to clusters randomly, i.e. we put each node in cluster $i$ with probability $\pi_i$ independently of all other. For further results regarding the block model in this sense see, for example, \cite{SBM1}.

The PPM is a generalization of classical Erdos-Renyi model, where each edge is drawn with probability $p$ independently from all others. In PPM model the appearance of the edges is still independent, but the probabilities might be different for different edges, which makes the model more versatile. 

One of the central questions about the PPM graphs is the question of restoring the original labels given an unlabeled PPM graph (see \cite{Abbe, Lei}). One of the natural ways to do it in case $p > q$ is to obtain a good clusterization of its vertices. In this paper we consider one of the most popular clustering heuristics, modularity, and obtain the bounds on its value for PPM graphs.

Modularity was first introduced in \cite{Item22} by Newman and Girvan. It appeared to be a good measure of how well can we split the graph into clusters in such a way that the edge density in clusters is high and the edge density between the clusters is low. It is known (see, for example, \cite{ModF1, ModF2}), that modularity has some flaws: for instance, it fails to detect small communities. Nevertheless, in the last two decades modularity became a key value for many clustering algorithms (\cite{Alg1, Alg2}) and was widely used in different applications including biology, physics, sociology and computer science. A lot of works in the past decade were dedicated to finding the modularity for various families of graphs (see \cite{Der, Item9a, Item11, McD, Trees, Rai}). 

Let us give a rigorous definition of the modularity. We will first introduce two auxiliary definitions.

\begin{definition} 
Let $\mathcal{A} = \{A_1, A_2, ..., A_k\}$ be a partition of the vertices of a graph $G$ with set of edges equal to $E_G$. The \textit{edge contribution} is then defined as $$\sum\limits_{i = 1}^{k}{\frac{e(A_k)}{e(G)}},$$ where $e(A) = |\{(v_1,v_2) \in E_G | v_1,v_2 \in A\}|$. 

Another (and slightly more convenient) way to write this value is the following one.
$$1 - \sum\limits_{i = 1}^{k}{\frac{e(A_i, \overline{A_i})}{2e(G)}},$$ where $e(A, B) = |\{(v_1,v_2) \in E_G | v_1 \in A, v_2 \in B\}|$ and $\overline{A} = G \setminus A$.

\end{definition}

\begin{definition}
Let $\mathcal{A} = \{A_1, A_2, ..., A_k\}$ be a partition of the vertices of a graph $G$. The value $$\sum\limits_{i = 1}^{k}{\frac{(\sum_{v\in A_k}{\deg(v)})^2}{4e^2(G)}}$$ is called the \textit{degree tax}.
\end{definition}

We are now ready to define the modularity of a partition.

\begin{definition}
Let $\mathcal{A} = \{A_1, A_2, ..., A_k\}$ be a partition of the vertices of a graph $G$. The \textit{modularity of partition} $\mathcal{A}$ is defined as the difference between the edge contribution and the degree tax of this partition:
\begin{equation*}
    q(\mathcal{A}) = \sum_{A \in \mathcal{A}}\frac{e(A)}{e(G)} -
    \sum_{A \in \mathcal{A}}\frac{(\sum_{v \in A}\deg(v))^2}{4e^2(G)}.
\end{equation*}
\end{definition}

Finally, we define the modularity of a graph as maximum over all its partitions.

\begin{definition}
The modularity of graph $G$ is defined as a maximum of modularities over all partitions of the vertices of $G$:
\begin{equation*}
    q^*(G) = \max_{\mathcal{A}}\{q(\mathcal{A})\}.
\end{equation*}
\end{definition}

In this paper, however, we need a slight generalization of the notion of the modularity. Let us consider a complete graph and write a non-negative number $w{(v_1, v_2)}$ (we will refer to it as weight) on each of its edges $(v_1, v_2)$. We can now define modularity just as above, but with slight modifications: $$e(A) := \sum\limits_{v_1,v_2 \in A}w{(v_1, v_2)},$$
$$
\deg(v) := \sum\limits_{v_1 \in G}w{(v, v_1)}.
$$

Note that this is a generalization of the previous definition, as we can put $w{(i, j)} = 1$ whenever there is an edge between $i$ and $j$ and $w{(i, j)} = 0$ if there is no edge between these vertices to obtain a classical definition.

Throughout this paper we consider PPM graphs with $n_1 = \dotsc = n_k = n$. We use the denotation $G(n, k, p, q)$ for them. We obtain bounds on the modularity of $G(n, k, p, q)$ graphs for different parameters $k, p$ and $q$. We also discuss how well does the eigengap heuristic perform on such graphs. 

The main result of the paper is the following bound on the modularity of $G(n, k, p, q)$.

\begin{theorem}
    \label{main}
    Let $n_i, k_i$ be a sequence of positive integers, and let $p_i, q_i \in [0, 1]$. Suppose $$n_ik_i \to \infty, p_i(n_i - 1) + q_i(n_ik_i - k_i) = \omega\left(\sqrt{n_ik_i}\right).$$
    Then almost surely
    \begin{multline*}
        \liminf\limits_{i \to \infty} \frac{(n_i - 1)p_i}{p_i(n_i - 1) + q_i(n_ik_i - k_i)} - \frac{1}{k_i}
        \leq \liminf\limits_{i \to \infty} q^{*}(G(n_i, k_i, p_i, q_i)) \\ \leq \liminf\limits_{i \to \infty} \frac{\max\{p_i, |(n_i - 1)p_i - n_iq_i|\}}{p_i(n_i - 1) + q_i(n_ik_i - k_i)}
    \end{multline*}
    and
    \begin{multline*}
        \limsup\limits_{i \to \infty} \frac{(n_i - 1)p_i}{p_i(n_i - 1) + q_i(n_ik_i - k_i)} - \frac{1}{k_i}
        \leq \limsup\limits_{i \to \infty} q^*(G(n_i, k_i, p_i, q_i)) \leq \\ \leq \limsup\limits_{i \to \infty} \frac{\max\{p_i, |(n_i - 1)p_i - n_iq_i|\}}{p_i(n_i - 1) + q_i(n_ik_i - k_i)}.
    \end{multline*}
\end{theorem}

Let us discuss the bounds obtained in this theorem. First of all, if $p$ is greater than $|(n - 1)p - nq|$, the theorem gives us that the modularity almost surely tends to zero. This case is not very interesting so we can consider only the other option. Secondly, for $p > q$, the case where we might hope to use the modularity, we can get rid of the modulo sign to obtain the following upper bound
\begin{multline*}
\limsup\limits_{i \to \infty} q^*(G(n_i, k_i, p_i, q_i)) \leq 
 \\ \leq \limsup\limits_{i \to \infty} \frac{(n_i - 1)p_i}{p_i(n_i - 1) + q_i(n_ik_i - k_i)} - \frac{n_iq_i}{p_i(n_i - 1) + q_i(n_ik_i - k_i)} \leq \\ \leq \limsup\limits_{i \to \infty} \frac{(n_i - 1)p_i}{p_i(n_i - 1) + q_i(n_ik_i - k_i)}.
\end{multline*}
 
We can now see, that when, for example, $k_i$ (that is, the number of clusters) tends to infinity, the difference between the lower and the upper bound tends to zero. Thus, the following corollary holds.

\begin{corollary}
    Let $n_i, k_i$ be a sequence of positive integers, and let $p_i, q_i \in [0, 1]$. Suppose $$k_i \to \infty, p_i(n_i - 1) + q_i(n_ik_i - k_i) = \omega\left(\sqrt{n_ik_i}\right), (n_i - 2)p_i > n_iq_i.$$
    Then almost surely
    \begin{equation*}
        \limsup\limits_{i \to \infty} q^*(G(n_i, k_i, p_i, q_i)) = \limsup\limits_{i \to \infty} \frac{(n_i - 1)p_i - n_iq_i}{p_i(n_i - 1) + q_i(n_ik_i - k_i)}
    \end{equation*}
    and
        \begin{equation*}
        \liminf\limits_{i \to \infty} q^*(G(n_i, k_i, p_i, q_i)) = \liminf\limits_{i \to \infty} \frac{(n_i - 1)p_i - n_iq_i}{p_i(n_i - 1) + q_i(n_ik_i - k_i)}.
    \end{equation*}
\end{corollary}

The paper is organized as follows. In the second section we introduce a new concept of $(n, d, \lambda)$ weighted graphs and discuss the properties of such graphs. In the third section we will prove the theorem that connects the modularity of $(n, d, \lambda)$ weighted graphs with the modularity of random graphs. Finally, in the forth section we apply these results to our model to get the bounds on the modularity of $G(n, k, p, q)$. 

\section{Regular weighted graphs}

\subsection{Definitions}

Consider a graph $G$ with vertex set $V$ and the weight function $w$ that assigns each unordered pair of vertices a non-negative weight (that is, $w: V^2 \to \mathbb{R}^{+}, w(i, j) = w(j, i)$). We call $G$ a $d$-regular weighted graph if the value 
$$
\deg(v) = \sum_{u \in G}w(v, u)
$$
equals $d$ for all $v$.

To define $(n, d, \lambda)$ weighted graphs, we, however, need to define the spectrum of a weighted graph.

\begin{definition}
    Consider a weighted graph $G$ on $n$ vertices. Let $W$ be the matrix such that $W_{i,j} = w(i, j)$. Let $\lambda_1 \geq \lambda_2 \dotsc \geq \lambda_n$ be the sequence of eigenvalues of $W$ (all of them are real, as $W$ is a symmetric matrix). We call this sequence {\it the spectrum} of $G$.
\end{definition}

It is obvious that the matrix $W$ is non-negative. Thus, due to Perron-Frobenius theorem, we have $\lambda_1 \geq \lambda := \max\{|\lambda_2|, \dotsc, |\lambda_n|\}$ and 
$$
\min_{i}\sum_{j = 1}^{n}W_{i, j} \leq \lambda_1 \leq \max_{i}\sum_{j = 1}^{n}W_{i, j}.
$$

In case of a $d$-regular weighted graph all sums $\sum_{j = 1}^{n}W_{i, j} = \deg(i)$ are equal to $d$. This gives us the desired equation $\lambda_1 = d$. We will call such graphs $(n, d, \lambda)$ weighted graphs.

In classical unweighted setting $(n, d, \lambda)$-graphs is a well studied topic. In the next subsection we attempt to transfer some of the results from the unweighted case to our new setting.

\subsection{Properties of $(n, d, \lambda)$-weighted graphs}

The first thing to do is to transfer the main lemma for unweighted $(n, d, \lambda)$-graphs (see \cite{Kriv06}) to weighted setting. 

\begin{lemma}\label{econc}
    Consider an $(n, d, \lambda)$ weighted graph and two (not necessarily disjoint) subsets of its vertices, $A$ and $B$. Let $e(A, B)$ be the sum of weights of edges between $A$ and $B$, where the weights of edges in $A \cap B$ are counted twice. Then
    $$
    \left|e(A, B) - \frac{d|A||B|}{n}\right| \leq \frac{\lambda}{n}\sqrt{|A||B|(n - |A|)(n - |B|)}.
    $$
\end{lemma}

\begin{proof}
    Let $\chi_{A}$ be the vector such that its $i$-th component is equal to 1 iff the $i$-th vertex belongs to $A$ and all the other components equal 0. Define $\chi_B$ similarly. Then $e(A, B) = \chi^{t}_AW\chi_B$. Let $x$ be a vector $\frac{1}{\sqrt{n}}(1, \dotsc, 1)$. It is straightforward to see that $Wx = dx$, and thus $x$ is an eigenvector with eigenvalue $\lambda_1 = d$. We can write $\chi_A = \alpha x + R_A, \chi_B = \beta x + R_B$, where $\alpha = \langle\chi_A,x\rangle = |A|/n, \beta = \langle\chi_B,x\rangle = |B|/n$. From linear algebra we know that $\langle x, R_B\rangle = \langle x, R_A\rangle = 0$. It is now easy to see that
    $$
    \chi^{t}_AW\chi_B = \alpha\beta x^tWx + R^t_AWR_b = \frac{d|A||B|}{n} + R^t_AWR_B.
    $$
    It remains to prove that $\|R^t_AWR_B\| \leq \frac{\lambda}{n}\sqrt{|A||B|(n - |A|)(n - |B|)}$. This follows from
    \begin{multline*}
        \|R^t_AWR_B\| \leq \|R_A\|\|WR_B\| \leq \lambda\|R_A\|\|R_B\| = \\ = \lambda\sqrt{|A| - |A|^2/n^2}\sqrt{|B| - |B|^2/n^2} = \frac{\lambda}{n}\sqrt{|A||B|(n - |A|)(n - |B|)}.
    \end{multline*}
\end{proof}

We can now use this lemma to prove an analogue of Theorem 4 from \cite{Item25}. 

\begin{lemma}\label{modb}
    For any $(n, d, \lambda)$ weighted graph $G$ we have
    $$
        q^*(G) \leq \frac{\lambda}{d}.
    $$
\end{lemma}

\begin{proof}
    Consider a partition $\mathcal{A} = \{A_1, \dotsc, A_k\}$ of $G$.
    We know that
    $$
    q(\mathcal{A}) = \sum_{i = 1}^{k}\frac{e(A_i)}{e(G)} - \sum_{i = 1}^{k}\frac{\left(\sum_{v \in A_i}\deg(v)\right)^2}{4e^2(G)} = 1 - \sum_{i = 1}^{k}\frac{e(A_i, \overline{A_i})}{dn} - \sum_{i = 1}^{k}\frac{\left(d|A|\right)^2}{(dn)^2},
    $$
    where $\overline{A_i} = G \setminus A_i$. Using Lemma \ref{econc}, we obtain
    \begin{multline*}
    q(\mathcal{A}) \leq 1 - \sum_{i = 1}^{k}\frac{(d - \lambda)|A|(n - |A|)}{dn^2} - \sum_{i = 1}^{k}\frac{|A|^2}{n^2} = \\ = 1 - \sum_{i = 1}^{k}\frac{d|A|n}{dn^2} + \sum_{i = 1}^{k}\frac{d|A|^2}{dn^2} + \sum_{i = 1}^{k}\frac{\lambda|A|n}{dn^2} - \sum_{i = 1}^{k}\frac{\lambda|A|^2}{dn^2} - \sum_{i = 1}^{k}\frac{|A|^2}{n^2} = \\ = \frac{\lambda}{d} - \sum_{i = 1}^{k}\frac{\lambda|A|^2}{dn^2} \leq \frac{\lambda}{d}.
    \end{multline*}
\end{proof}

\section{Regular weighted graphs and random graphs}

\subsection{Main result}

The main result of this section connects the modularity of a $(n, d, \lambda)$ weighted graph with weights in $[0, 1]$ with the modularity of a corresponding random graph. Let us elaborate on the meaning of the word corresponding.

\begin{definition}
    Let $H$ be a weighted graph on $n$ vertices with $w(i, j) \in [0, 1]$ for all $i, j$. We define the corresponding random graph as a graph on $n$ vertices where each edge $(i, j)$ is drawn with probability $w(i, j)$ independently from all others. 
\end{definition}

We can now state the main theorem.

\begin{theorem}\label{ProbMain}
    Let $H_n$ be a sequence of $(n, d_n, \lambda_n)$-weighted graphs with all weights in $[0, 1]$, and let $H^{\prime}_n$ be the random graph corresponding to $H_n$ for each $n$. Suppose that 
    $$
        d_n = \omega(n^{1/2}), \limsup\limits_{n \to \infty} \frac{\lambda_n}{d_n} < 1,
    $$ 
    then
    $$
        \left|q^{*}(H_n) - q^{*}(H^{\prime}_n)\right| \toas 0, n \to \infty.
    $$
\end{theorem}

\begin{remark}
    In this section we consider different graphs, so the denotations like $\deg(v), e(A, B)$ and so on might be ambiguous. To deal with it we write the graph that we refer to as a subscript, e.g. $\deg_H(v), e_{H^{\prime}}(A, b), \dotsc$. 
\end{remark}

In the proof we use the following classical inequality proved in \cite{Hef}.

\begin{lemma}[Hoeffding]\label{Hf}
    \text{ } \\
    Let $X_1, \dotsc, X_m$ be independent random variables such that for each $i$ we have $\mathbf{P}(X_i \in [a_i, b_i]) = 1$ for some $a_i$ and $b_i$. Let $S_{m} = X_1 + \dotsc + X_m$. Then the following inequality holds.
    $$
    \mathbf{P}(|S_{m} - \mathbf{E}S_m| \geq t) < 2\exp\left(-\frac{2t^2}{\sum\limits_{i = 1}^{m}(b_i - a_i)^2}\right).
    $$
\end{lemma}

We also need the following trivial lemma, which follows easily from Borel–Cantelli lemma.

\begin{lemma}\label{BC}
    Let $Q_n$ be a sequence of events such that $\mP(Q_n) \leq e^{-Cn}$ for some positive $C$. Then almost surely only a finite number of $Q_n$ occur.
\end{lemma}

Finally, we need an algorithm to modify partition without changing its modularity much.

\begin{construction}\label{algorithm}
    Let $G_n$ be a graph on $n$ vertices, and let $k \geq 2$, $k \in \mathbb{N}$. Let us also fix some positive integer $k$.
    We define an algorithm to construct two sets, $\mathcal{A}^{\prime}$ and $\mathcal{A}_{big}$ from a partition $\mathcal{A} = \{A_1, A_2, \ldots, A_m\}$. We also need an auxiliary set $A_{merged}$ which is initially empty.
    
    Iterate through $A_i, i\in\{1,2,\ldots,m\}$ :
    \begin{itemize}
        \renewcommand{\labelitemi}{\textbullet}
        \renewcommand{\labelitemii}{\textperiodcentered}
        \item If $|A_i| > \frac{n}{k}$, then add $A_i$ to $A^{\prime}$
        \item  If $|A_i| \leq \frac{n}{k}$ :
        \begin{itemize}
            \item $A_{merged} := A_{merged} \cup A_i$
            \item If $|A_{merged}| \geq \frac{n}{k}$, then $A_{merged}$ is added to $\mathcal{A}^{\prime}$ and $A_{merged} := \varnothing$
        \end{itemize}
    \end{itemize}
    
    After iterating through $A_i \in \mathcal{A}$ consider the set $A_{merged}$ :
    \begin{itemize}
        \renewcommand{\labelitemi}{\textbullet}
        \item If $|A_{merged}| > \frac{n}{k}$, then $A_{merged}$ is added to $\mathcal{A}^{\prime}$
        \item If $|A_{merged}| \leq \frac{n}{k}$, then $A_{merged}$ is ignored
    \end{itemize}
    
    Finally, $\mathcal{A}_{big}$ is defined in the following way:
    $$
        \mathcal{A}_{big} := \bigg\{A : A \in \mathcal{A}^{\prime}, |A| > \frac{2n}{k}\bigg\}.
    $$
\end{construction}

\begin{remark}
Note the following key property of this algorithm: any element from $\mathcal{A}_{big}$ is also present in $\mathcal{A}$ and any element from $\mathcal{A} \setminus \mathcal{A}_{big}$ has size of no more than $\frac{2n}{k}$.
\end{remark}

We give the proof of Theorem \ref{ProbMain} is subsection \ref{mproof}. In subsection \ref{pares} we formulate and prove auxiliary concentration inequalities. 

Note that the proofs in this section are almost the same as in section ``Stability of the modularity" in \cite{JGS}. We will however give them here for the sake of completeness.

\subsection{Auxiliary lemmas} \label{pares}

In this section we prove concentration inequalities on $e(H^{\prime}_n), e_{H^{\prime}_n}(V, \overline{V})$ and $\sum\limits_{v \in V}{\deg_{H^{\prime}_n}(v)}$, where $V$ is a subset of vertices of $H^{\prime}_n$ (with some restrictions on its size).

\begin{lemma}
    \label{LeG}
    Consider the sequence $H^{\prime}_n$ from Theorem \ref{ProbMain}. Then almost surely the following inequality holds for all sufficiently large $n$. 
    \begin{equation}
    \label{L1}
    \left|e(H^{\prime}_n) - \frac{d_nn}{2}\right| < \sqrt{\frac{n^{1/2}}{d_n}}\frac{d_nn}{2},
    \end{equation}
    where $d_n$ is the degree of each vertex in corresponding weighted graph $H_n$.
\end{lemma}

\begin{proof}
From Lemma \ref{Hf} for any fixed $n$ we have 
$$
\mP\left(\left|e(H^{\prime}_n) - \frac{d_nn}{2}\right| \geq \sqrt{\frac{n^{1/2}}{d_n}}\frac{d_nn}{2}\right) \leq 2\exp\left\{-n^{1/2}d_n\right\}.
$$

Since $d_n = \omega(\sqrt{n})$, we get that there exists a positive constant $C$ such that $$\mP\left(\left|e(H^{\prime}_n) - \frac{d_nn}{2}\right| \geq \sqrt{\frac{n^{1/2}}{d_n}}\frac{d_nn}{2}\right) \leq e^{-Cn}.$$ By applying Lemma \ref{BC} we obtain the desired result.
\end{proof}

\begin{lemma}
    \label{LeA}
        Consider the sequence $H^{\prime}_n$ from Theorem \ref{ProbMain} and let $k$ be a positive integer. Then almost surely the following inequality holds for all sufficiently large $n$.
        $$\forall V \subset V_{H^{\prime}_n}, |V| \in \left[\frac{n}{k}, \frac{(k - 1)n}{k}\right] : |e_{H^{\prime}_n}(V, \overline{V}) - e_{H_n}(V, \overline{V})| < \sqrt{\frac{n^{1/2}}{d_n}}e_{H_n}(V, \overline{V}).$$
\end{lemma}

\begin{proof}
Let us fix some $n$ and bound the probability of the following event for some fixed  $V \subset V_G$: $|e_{H^{\prime}_n}(V, \overline{V}) - e_{H_n}(V, \overline{V})| \geq \sqrt{\frac{n^{1/2}}{d_n}}e_{H_n}(V, \overline{V})$. Taking $t = \sqrt{\frac{n^{1/2}}{d_n}}e_{H_n}(V, \overline{V})$ and applying Lemmas \ref{Hf} and \ref{econc} we obtain 
\begin{multline*}
\mathbf{P}\left(|e_{H^{\prime}_n}(V, \overline{V}) - e_{H_n}(V, \overline{V})| \geq \sqrt{\frac{n^{1/2}}{d_n}}e_{H_n}(V, \overline{V})\right) < 2\exp\left\{-2\frac{n^{1/2}}{d_n}\frac{e^2_{H_n}(V, \overline{V})}{|V||\overline{V}|}\right\} \leq \\ 
\leq 2\exp\left\{-2\frac{n^{1/2}}{d_n}\frac{(d_n-\lambda_n)^2|V||\overline{V}|}{n^2}\right\}.
\end{multline*}
As we know that for some positive $C$ and big enough $n$ we have $d_n - \lambda_n \geq Cd_n$, we get that for all $V : |V| \in \left[\frac{n}{k}, \frac{(k - 1)n}{k}\right]$ the following bound holds.
$$
\mathbf{P}\left(|e_{H^{\prime}_n}(V, \overline{V}) - e_{H_n}(V, \overline{V})| \geq \sqrt{\frac{n^{1/2}}{d_n}}e_{H_n}(V, \overline{V})\right) < 2\exp\left\{-2\frac{n^{1/2}}{d_n}\frac{C^2d_n^2}{k^2}\right\},
$$
or, equivalently,
$$
\mathbf{P}\left(|e_{H^{\prime}_n}(V, \overline{V}) - e_{H_n}(V, \overline{V})| \geq \sqrt{\frac{n^{1/2}}{d_n}}e_{H_n}(V, \overline{V})\right) < 2\exp\left\{-2C_0(k)d_n\sqrt{n}\right\}.
$$
That means that the probability the property we want to prove does not hold is no more than
$$
2^n\exp\left\{-C_0d_n\sqrt{n}\right\} \leq \exp\left\{n\ln 2 - C_0d_n\sqrt{n}\right\} < e^{-C_1d_n\sqrt{n}}.
$$
Since $d_n = \omega(n^{1/2})$, it remains to apply Lemma \ref{BC}.
\end{proof}

\begin{lemma}
    \label{Ldeg}
    Consider the sequence $H^{\prime}_n$ from Theorem \ref{ProbMain} and let $k$ be a positive integer. Then almost surely the following inequality holds for all sufficiently large $n$. 
    \begin{multline*}
    \forall V \subset V_{H^{\prime}_n}, |V| \geq \frac{n}{k} : \left|\sum_{v \in V}{\deg_{H^{\prime}_n}(v)} - \sum_{v \in V}{\deg_{H_n}(v)}\right| < \sqrt{\frac{n^{1/2}}{d_n}}\sum_{v \in V}{\deg_{H_n}(v)}  = \\ = \sqrt{\frac{n^{1/2}}{d_n}}d_n|V|.
    \end{multline*}
\end{lemma}

\begin{proof}
Let us fix some $n$ and bound the probability of the following event
$$
\left|\sum_{v \in V}{\deg_{H^{\prime}_n}(v)} - \sum_{v \in V}{\deg_{H_n}(v)}\right| \geq \sqrt{\frac{n^{1/2}}{d_n}}\sum_{v \in V}{\deg_{H_n}(v)}
$$ 
for some fixed $V \subset V_{H_n}$. First let us rewrite $\sum_{v \in V}{\deg(v, p)}$ in form $\sum\limits_{e \in E_1}{X_e} + \sum\limits_{e \in E_2}{2X_e}$, where $E_1$ is the set of pairs of vertices from $V \times \overline{V}$, $E_2$ is the set of pairs of vertices from $V \times V$, and $X_e$ is an indicator of an edge $e$ from $H_n$ appearing in $H^{\prime}_n$. Then, applying Lemma \ref{Hf} with $$a_i = 0, b_i = 2, t = \sqrt{\frac{n^{1/2}}{d_n}}\sum_{v \in V}{\deg_{H_n}(v)},$$ we get
\begin{multline*}
    \mathbf{P}\left(\left|\sum_{v \in V}{\deg_{H^{\prime}_n}(v)} - \sum_{v \in V}{\deg_{H_n}(v)}\right| \geq \sqrt{\frac{n^{1/2}}{d_n}}\sum_{v \in V}{\deg_{H_n}(v)}\right) < \\ < 2\exp\left\{-\frac{2n^{1/2}\left(\sum\limits_{v \in V}\deg_{H_n}(v)\right)^2}{4d_n|V|n}\right\} = 2\exp\left\{-\frac12\frac{d_n|V|}{n^{1/2}}\right\}.
\end{multline*}

Considering only $V : |V| \geq \frac{n}{k}$ we obtain the following bound for such $V$:
\begin{multline*}
\mathbf{P}\left(\left|\sum_{v \in V}{\deg(v, p)} - p\sum_{v \in V}{\deg(v)}\right| \geq \frac{p\log(f(n))}{f(n)}\sum_{v \in V}{\deg(v)}\right) < \\ 2\exp\left\{-\frac12\frac{d_n|V|}{n^{1/2}}\right\} \leq \exp\left\{-C(k)d_nn^{1/2}\right\}.
\end{multline*}

The probability that the property does not hold for at least one such $V$ can then be easily bounded from above:
$$
2^{n}\exp\left\{-Cn\log^2(f(n))\right\} \leq \exp\left\{n\ln2 - Cd_nn^{1/2}\right\} < e^{-C_1d_nn^{1/2}}.
$$
Since $d_n = \omega(n^{1/2})$, it remains to apply Lemma \ref{BC}.

\end{proof}

\subsection{Proof of the main result} \label{mproof}

Let us prove that almost surely $\limsup_{n \to \infty} |q^*(G_{n}) - q^{*}(G_{n}(p))| < \varepsilon$ for any $\varepsilon > 0$, which is obviously equivalent to the statement of the theorem. Since the properties from Lemmas \ref{LeG} --- \ref{Ldeg} hold almost surely, we can assume that the bounds from these lemmas always hold. 

Let us first prove the upper bound, $q^{*}(H_n) \leq q^{*}(H^{\prime}_n) + \varepsilon$.

Let us fix some $k$ (we will specify its value later). Consider an optimal partition $\mathcal{A}$ of graph $H_n$. Let us first assume that each of the sets in $\mathcal{A}$ has less than $\frac{(k - 2)n}{k}$ elements. Let us then apply Construction \ref{algorithm} to $\mathcal{A}$ to obtain $\mathcal{A}^{\prime}$ and $\mathcal{A}_{big}$. Note that while $\mathcal{A}^{\prime}$ is not a partition, there exists a set $U$ such that $|U| < \frac{n}{k}$ and $\mathcal{A}^{\prime} \cup \{U\}$ is a partition. Let us add the elements of $U$ to one of the sets (we will call this set $V$) from $\mathcal{A}^{\prime}$. So we obtain some partition $\mathcal{A}^{\prime\prime}$. Then we can get the following bound on $q_{H_n}(\mathcal{A}^{\prime\prime})$.
\begin{multline*}
q_{H_n}(\mathcal{A}^{\prime\prime}) = \sum\limits_{A \in \mathcal{A}^{\prime\prime}}{\frac{e_{H_n}(A)}{e(H_n)}} - \sum\limits_{A \in \mathcal{A}^{\prime\prime}}{\frac{d_n^2|A|^2}{4e^2(H_n)}} = \sum\limits_{A \in \mathcal{A}_{big}\setminus\{V\}}{\frac{e_{H_n}(A)}{e(H_n)}} + \sum\limits_{A \in \mathcal{A}^{\prime\prime}\setminus\mathcal{A}_{big}}{\frac{e_{H_n}(A)}{e(H_n)}} - \sum\limits_{A \in \mathcal{A}^{\prime\prime}}{\frac{|A|^2}{n^2}} \geq \\ \geq \sum\limits_{A \in \mathcal{A}}{\frac{e_{H_n}(A)}{e(H_n)}} - \sum\limits_{A \in \mathcal{A}^{\prime\prime}}{\frac{|A|^2}{n^2}} \geq \sum\limits_{A \in \mathcal{A}}{\frac{e_{H_n}(A)}{e(H_n)}} - \sum\limits_{A \in \mathcal{A}_{big}\setminus\{V\}}{\frac{|A|^2}{n^2}} - \sum\limits_{A \in \mathcal{A}^{\prime}\setminus \mathcal{A}_{big}}{\frac{|A|^2}{n^2}} - \\ - \frac{|V|^2}{n^2} + \frac{|V^2|}{n^2} - \frac{|U\cup V|^2}{n^2} \geq \sum\limits_{A \in \mathcal{A}}{\frac{e_{H_n}(A)}{e(H_n)}} - \sum\limits_{A \in \mathcal{A}}{\frac{|A|^2}{n^2}} + \frac{|V|^2}{n^2} - \frac{|U\cup V|^2}{n^2} -\\- \sum\limits_{A \in \mathcal{A}^{\prime} \setminus \mathcal{A}_{big}}{\frac{|A|^2}{n^2}} = q_{H_n}(\mathcal{A}) +\frac{|V|^2}{n^2} - \frac{(|U| + |V|)^2}{n^2} - \sum\limits_{A \in \mathcal{A}^{\prime} \cup \{U\} \setminus \mathcal{A}_{big}}{\frac{|A|^2}{n^2}} \geq \\ \geq q_{H_n}(\mathcal{A}) - \frac4k - \frac1{k^2} = q^{*}(H_n) - \frac4k - \frac1{k^2}.
\end{multline*}

Now consider $\mathcal{A}^{\prime\prime}$ as a partition of a random graph $H^{\prime}_n$. Note that
$
q^{*}(H^{\prime}_n) \geq q(\mathcal{A}^{\prime\prime}).
$
Let us rewrite $q_{H^{\prime}_n}(\mathcal{A}^{\prime\prime})$ as follows using Lemmas \ref{LeG} --- \ref{Ldeg}.
\begin{equation*}
q_{H^{\prime}_n}(\mathcal{A}^{\prime\prime}) = \left(\sum\limits_{A \in \mathcal{A}^{\prime\prime}}{\frac{e_{H_n}(A)}{e(H_n)}} - \sum\limits_{A \in \mathcal{A}^{\prime\prime}}{\frac{\left(\sum_{v\in A}{\deg_{H_n}(v)}\right)^2}{4e^2(H_n)}}\right)(1 + o(1)) = q_{H_n}(\mathcal{A}^{\prime\prime})(1 + o(1)).
\end{equation*}
We thus obtain $q_{H^{\prime}_n}(\mathcal{A}^{\prime\prime}) \geq (1 - o(1))(q^{*}(H_n) - \frac4k - \frac1{k^2})$ and thus $q^{*}(H_n) \leq q^{*}(H^{\prime}_n) + \frac4k + \frac1{k^2} + o(1)$.

Let us now assume that $\mathcal{A}$ has a set of size at least $\frac{(k - 2)n}{k}$. In this case the modularity of $H_n$ does not exceed $\frac{4k - 4}{k^2} \leq \frac4k$.
At the same time the modularity of $H^{\prime}_n$ is at least 0, so in this case we have $q^{*}(H_n) \leq q^{*}(H^{\prime}_n) + \frac4k$. 

It remains to notice that taking $k$ such that $ \frac4k + \frac1{k^2} < \varepsilon$ completes the proof of the first part.

Let us now prove the second part of the theorem, $q^{*}(H_n) \geq q^{*}(H^{\prime}_n) - \varepsilon$.

Consider an optimal partition $\mathcal{A}$ of graph $H^{\prime}_n$. Assume it contais a set $A_i$ such that $|A_i| > \frac{(k - 1)n}{k}$. The modularity of $H^{\prime}_n$ is then bounded by $$1 - \frac{|A_i|^2}{n^2}(1 + o(1)) \leq \frac2{k} + o(1).$$ Now assume that there are no such $A_i$ in $\mathcal{A}$. Let us then apply Construction \ref{algorithm} to obtain $\mathcal{A}^{\prime}$ and $\mathcal{A}_{big}$ from $\mathcal{A}$. Just like in the previous part, let $U$ be the set of vertices that did not fall into any of the sets of $\mathcal{A}^{\prime}$. It is obvious that $q^{*}(H^{\prime}_n)$ does not exceed  
$$
    	1 
    	- \frac12\sum\limits_{A \in \mathcal{A}^{\prime}}{\frac{e_{H^{\prime}_n}(A, \overline{A})}{e(H^{\prime}_n)}} - \sum\limits_{A \in \mathcal{A}_{big}}\frac{\left(\sum_{v \in A}{\deg_{H^{\prime}_n}(v)}\right)^2}{4e^2(H^{\prime}_n)}.
$$
Applying Lemmas \ref{LeG} -- \ref{Ldeg} we obtain that for all big enough $n$
\begin{multline*}
    	q^{*}(H^{\prime}_n) \leq  \left(1 - \frac12\sum\limits_{A \in \mathcal{A}^{\prime}}{\frac{e_{H_n}(A, \overline{A})}{e(H_n)}} - \sum\limits_{A \in \mathcal{A}_{big}}\frac{|A|^2}{n^2}\right)(1 + o(1)) = \\ = \left(1 - \frac12\sum\limits_{A \in \mathcal{A}^{\prime}\cup \{U\}}{\frac{e_{H_n}(A, \overline{A})}{e(H_n)}} + \frac{e_{H_n}(U,\overline{U})}{dn} - \sum\limits_{A \in \mathcal{A}_{big}}\frac{|A|^2}{n^2} - \sum\limits_{A \in \mathcal{A}^{\prime} \cup \{U\} \setminus \mathcal{A}_{big}}\frac{|A|^2}{n^2} + \sum\limits_{A \in \mathcal{A}^{\prime} \cup \{U\} \setminus \mathcal{A}_{big}}\frac{|A|^2}{n^2}\right)\times \\ \times(1 + o(1)) \leq \left(q^{*}(H_n) + \frac{e_{H_n}(U, \overline{U})}{d_nn} + \sum\limits_{A \in \mathcal{A}^{\prime} \cup \{U\} \setminus \mathcal{A}_{big}}\frac{|A|^2}{n^2}\right)(1 + o(1)) \leq \\ \leq q^{*}(H_n) + \frac{d_n + \lambda_n}{dk^2} + \frac{2}{k} + o(1) \leq q^{*}(H_n) + \frac{2}{k^2} + \frac{2}{k} + o(1).
\end{multline*}

It remains to notice that taking any $k$ such $ \frac2k + \frac2{k^2} < \varepsilon$ completes the proof of the second part of the theorem.

\section{Application of the results. Proof of Theorem \ref{main}}

As follows from the previous sections, to obtain bounds on the modularity of $G(n, k, p, q)$, we need to consider the following family of weighted graphs, that are corresponding graphs for $G(n, k, p, q)$.

\begin{definition}
    Given two integers, $n$ and $k$ and two real numbers $p, q \in [0, 1]$, define the following weighted graph $H(n, k, p, q)$: 
    $$
    V = \{0, \dotsc, nk - 1\}, w_{(i, j)} = 
    \begin{cases}
        0,~i=j\\
        p,~i \neq j, \lfloor i/n \rfloor = \lfloor j/n \rfloor \\
        q,~\lfloor i/n \rfloor \neq \lfloor j/n \rfloor
    \end{cases}.
    $$
\end{definition}

\begin{remark}
    Consider any vertex $v$ of an $H(n, k, p, q)$ graph. Its degree is equal to $d = (n - 1)p + (nk-n)q$ and does not depend on $v$.
\end{remark}

Let us now calculate the spectrum of $H(n, k, p, q)$ graphs.

\begin{theorem}\label{spech}
    $H(n, k, p, q)$ graph has the following eigenvalues: 
    \begin{enumerate}
        \item $d = (n - 1)p + (nk - n)q$ with multiplicity 1;
        \item $-p$ with multiplicity $nk - k$;
        \item $(n - 1)p - nq$ with multiplicity $k - 1$;
    \end{enumerate}
\end{theorem}

\begin{remark}
    Note that in some cases some of these eigenvalues might coincide. Namely, when $p = q$, the second and the third eigenvalues are equal. In this case the graph has 2 eigenvalues: $d$ with multiplicity 1 and $-p$ with multiplicity $nk - 1$. In case $q = 0$, the first and the third eigenvalues are equal and thus the graph has 2 eigenvalues: $d$ with multiplicity $k$ and $-p$ with multiplicity $nk - k$. Finally, when both $p$ and $q$ are equal to 0, all of these values are equal and the graph has one eigenvalue 0 of multiplicity $nk$. It is easy to see that in all the other cases all the values are different. In the proof we consider only the general case as the degenerate cases are easy to handle.
\end{remark}

\begin{proof}
    The first part is easy, since $H(n, k, p, q)$ is a regular weighted graph, so let us focus on the other two parts.

    Consider the matrix $W + pI$, where $I$ is the identity matrix and consider its rows corresponding to vertices $in, \dotsc, (i + 1)n - 1$ for some $i$ from $0$ to $k - 1$. It is easy to see that all these rows are equal. Thus the matrix rank does not exceed $k$, and the eigenvalue $-p$ has multiplicity at least $nk - k$. It remains to prove that the rank $W - ((n - 1)p - nq)I$ does not exceed $n - k + 1$. This follows from the fact that the sum of the rows corresponding to vertices $0$ through $n - 1$ equals the sum of the rows corresponding to vertices $in$ through $(i + 1)n - 1$ for each $i$ from 1 to $k - 1$.
\end{proof}

Finally we can apply Theorems \ref{modb} and \ref{spech} to derive the following theorem.

\begin{theorem}
    \label{dmain}
    Let $n, k$ be positive integers and let $p, q$ be real numbers in range $[0, 1]$ such that $p^2 + q^2 > 0$. Then 
    $$
    \frac{(n - 1)p}{(n - 1)p + (nk - n)q} - \frac{1}{k} \leq q^{*}(H(n, k, p, q)) \leq \frac{\max\{p, |(n - 1)p - nq|\}}{(n - 1)p + (nk - n)q}.
    $$
\end{theorem}

\begin{proof}
    The upper bound follows directly from Theorems \ref{modb} and \ref{spech}. Let us now prove the lower bound. Consider the natural partition $\mathcal{A} = \{A_1, \dotsc, A_k\}$, where $A_i = \{(i - 1)n, \dotsc, in - 1\}$. Its modularity is equal to
    $$
    \sum_{i = 1}^{k} \frac{n(n - 1)p}{dnk} - \frac{1}{k^2} = \frac{(n - 1)p}{d} - \frac{1}{k}.
    $$
\end{proof}

Theorem \ref{main} now follows directly from Theorems \ref{dmain} ans \ref{ProbMain}.

\end{document}